\documentclass[a4paper,12pt]{article}
\usepackage{mathrsfs}

\usepackage{amsfonts}
\usepackage{amscd,color}
\usepackage{amsmath,amsfonts,amssymb,amscd}
\usepackage{indentfirst,graphicx,epsfig}
\input{epsf}

\usepackage{epstopdf}
\usepackage{caption}

\usepackage[dvipsnames]{xcolor}
\usepackage[T1]{fontenc}
\usepackage[utf8]{inputenc}
\usepackage[english]{babel}
\usepackage{amsmath,amsthm,enumerate}

\usepackage[colorlinks=false, linktocpage=true]{hyperref}

\usepackage{url}
\usepackage{stmaryrd}
\usepackage{tikz,pgf}
\usepackage{subfig}
\usepackage{cancel}
\usetikzlibrary{positioning}
\usepackage{array}
\usepackage{makecell}
\usepackage{tabularx}
\usepackage{comment}

\usetikzlibrary{decorations.markings}
\usetikzlibrary{decorations.pathmorphing}
\usetikzlibrary{patterns}
\tikzset{->-/.style={decoration={
			markings,
			mark=at position .5 with {\arrow[scale=0.8]{>}}},postaction={decorate}}}
\tikzset{snake it/.style={decorate, decoration={snake, amplitude=.4mm, segment length=2mm}}}

\newtheorem{question}{Question}

\newtheorem{theorem}{Theorem}
\newtheorem{claim}{Claim}[section]
\newtheorem{proposition}{Proposition}
\newtheorem{lemma}{Lemma}
\newtheorem{case}{Case}
\newtheorem{property}{Property}
\newtheorem{observation}{Observation}
\usepackage{tikz}
\usepackage{enumitem}
\usepackage{algorithm}  
\usepackage{algpseudocode}  
\usetikzlibrary{shapes.geometric, arrows}
\newcommand{\ml}{l\kern-0.55mm\char39\kern-0.3mm}

\def\qed{\hfill \nopagebreak\rule{5pt}{8pt}}

\title{\textbf{On independent domination and packing numbers of subcubic graphs}}
\author{Xuqing Bai$^a$, \quad Zhipeng Gao$^a$\thanks{Corresponding author.}, \quad Changqing Xi$^b$, \quad Jun Yue$^c$\\
{\small $^a$ School of Mathematics and Statistics, Xidian University, Xi'an, 710071, China}\\
{\small $^b$ Center for Combinatorics and LPMC, Nankai University,
Tianjin, 300071, China}\\
{\small $^c$ School of Mathematics Science, Tiangong University,
Tianjin, 300387, China}\\
{\small $\{$baixuqing@xidian.edu.cn, gaozhipeng@xidian.edu.cn, 
xcqmath@163.com, yuejun06@126.com$\}$}}
\date{}

\begin{document}

\maketitle	
	
	\maketitle	
	\begin{abstract}
		In a recent paper, Cho and Kim proved that in subcubic graphs, the independent domination number is at most three times the packing number. They subsequently posed the question of characterizing subcubic graphs that achieve this bound. In this paper, we completely solve the question by proving that exactly four graphs meet this bound. 
	\end{abstract}
	
	
	\section{Introduction}
	In this paper, all graphs are finite, simple, and undirected. We define a graph $G$ as $(V(G), E(G))$ with a vertex set $V(G)$ and an edge set $E(G)$. The order of the graph is defined as $n = |V(G)|$. Let $u$ and $v$ be two vertices in $G$, and let $S$ be a subset of vertices in $G$. The length of a shortest path that connects $u$ and $v$ is called the \emph{distance} between $u$ and $v$ in $G$ and denoted $d_G(u,v)$. While the distance between $u$ and $S$ is given by $d_G(u,S)=\min\limits_{v\in S}\{d_G(u,v)\}$.
	For any integer $i \geq 1$, we denote the set of vertices in $G$ that are exactly at distance $i$ from $v$ as $N_i(v)$.
	In particular, $N_1(v)$ represents the open neighborhood, also denoted as $N_{G}(v)$, of $v$ in $G$. 
	Let $N_i[v]$ denotes the set of vertices within distance $i$ from $v$ in $G$ and $N_1[v]$ is the closed neighborhood of $v$ in $G$, also denoted as  $N_G[v]$. 
	Generalizing these concepts, the open neighborhood of $S$, denoted as $N_G(S)$, is $\bigcup_{v \in S} N_G(v)$, while the closed neighborhood of $S$, denoted by $N_G[S]$, is $N_G(S) \cup S$. 
	For $i \geq 1$, the $i^{\text {th}}$ \emph{boundary} of $S$ in $G$, denoted as $\partial_i(S)$, is the set of all vertices at distance $i$ from the set $S$ in $G$. 
	The $1^{\text {st }}$ boundary of $S$ in $G$ is also denoted by $\partial_{G}(S)$, so $\partial_{G}(S)=\partial_1(S)=N_G[S] \backslash S$. 
	Let $N_i[S]$ be the set of vertices within distance $i$ in $G$ from $S$. 
	Specially, $N_1[S]=N_G[S]$ and $N_2[S]=S \cup \partial_1(S) \cup \partial_2(S)$. 
	The \emph{degree} of $u$ in $G$ is $d_G(u)=|N_G(u)|$, while the minimum and maximum degrees of $G$ are denoted by $\delta(G)$ and $\Delta(G)$, respectively.
	The \emph{subgraph induced} by $S$ in $G$ is denoted by $G[S]$. 
	The path, cycle, and complete graph of order $n$ are denoted by $P_n$, $C_n$, and $K_n$, respectively. 
	The \emph{girth} of $G$, denoted by $g(G)$, is the minimum order of a cycle in $G$. 
	A \textit{subcubic graph} is a graph of maximum degree at most 3.
	The \emph{square} of a graph $G$, denoted as $G^2$, is the graph with $V(G^2)=V(G)$ and $uv\in E(G^2)$ if $1\le d_G(u,v)\le 2$. All subscripts, $G$, can be omitted when there is no ambiguity. Let $[n]=\{1,2,\ldots,n\}$. The notations not involved here are referred to \cite{haynes2023domination}. 
	
	An \emph{independent set} of $G$ is a set of vertices where no two are
	adjacent. The cardinality of a maximum independent set in $G$
	is called the independent number of $G$ and is denoted by $\alpha(G)$.
	Let $D$ be a vertex set of $G$. We say that $D$ \emph{dominates} a vertex $u$ of $G$ if $N_G[u]$ contains an element of $D$. The set $D$ is a \emph{dominating set} of $G$ if $D$ dominates every vertex of $G$. The minimum cardinality of a dominating set of $G$ is the \emph{domination number} of $G$, denoted as $\gamma(G)$.
	
	An \emph{independent dominating set} of $G$ is a vertex set that is both an independent set and a dominating set of $G$.
	Analogously, the \emph{independent domination number} of $G$, denoted by $i(G)$, is defined as the size of a minimum independent dominating set of $G$. 
	Independent dominating sets can be understood as maximal independent sets, and the independent domination number is the size of minimum maximal independent set.
	The study of the independent domination numbers of graphs has been extensive since the 1960s, as shown in \cite{berge1962theory, cho2023independent3, ore1962theory}. For an overview of the independent domination numbers of graphs, see \cite{goddard2013independent}. 
	It's clear that $i(G)\geq \gamma(G)$ for every graph $G$, and there have been many results comparing $i(G)$ to $\gamma(G)$. See, for example, \cite{cho2023tight, cho2023independent2}. 
	
	A vertex set $S$ of a graph $G$ is a \emph{packing} of $G$ if the closed neighborhoods of every two distinct vertices in $S$ are disjoint or, equivalently, $d_G(u, v)\geq 3$ for every two distinct vertices $u$ and $v$ in $S$. A packing of $G$ is \emph{maximal} if it is not the subset of any larger packing of $G$. The cardinality of a maximum packing of $G$ is the packing number of $G$, denoted as $\rho(G)$. The study of packing related to a dominating set in graphs has a long history, which dates back to the 1970s \cite{Frank1970Covering, meir1975relations}. Haynes et al. \cite{haynes2023domination} proved $\gamma(G)\geq \rho(G)$ for every graph $G$. 
	See also \cite{haynes1998fundamentalsofdomination, henning2011dominating, TOPP1991229} for an overview of research in this direction. In 1996, Favaron \cite{favaron1996signed} showed the following result on packing number when dealing with a signed domination problem. 
	\begin{lemma}\label{Favathm}{\upshape\cite{favaron1996signed}}		
		If $G$ is a connected cubic graph of order $n$, and $G$ is not the Petersen graph, then $\rho(G)\ge \frac{n}{8}$.
	\end{lemma}
	In 2017, Henning \cite{2017Packing} further generalized this result to $k$-regular graphs with $k \geq 3$.
	\begin{lemma}{\upshape\cite{2017Packing}}	\label{regular}
		If $G$ is a connected $k$-regular graph of order $n$ that is not a Moore graph of diameter $2$, then $\rho(G) \ge \frac{n}{k^2-1}$.
	\end{lemma}
	The bound given in Lemma \ref{regular} is sharp when $k=4$ and $k=5$. For 3-regular graphs, we obtain a result as stated in the following theorem. This result may be of independent interest, and we utilize it during the proof of Theorem \ref{th1}.
	\begin{theorem}\label{packingtheorem}
		If $G$ is a cubic graph of girth $5$ and order $n$, where $n\geq12$, then $\rho(G)\ge \frac{n+1}{8}$. 
	\end{theorem}
	Henning et al. \cite{henning2011dominating} observed that  $\gamma(G) \le \Delta(G)\rho (G)$, and characterized the connected subcubic graphs with $\gamma(G)=3\rho (G)$. Recently, Cho and Kim obtained the following stronger conclusion for subcubic graphs.
	
	\begin{theorem}{\upshape\cite{cho2023independent}}\label{th0}
		If $G$ is a subcubic graph, then $i(G)\leq 3\rho (G)$. 
	\end{theorem}
	Furthermore, Cho and Kim posed the following question.
	
	\begin{question}{\upshape\cite{cho2023independent}}\label{q}
		If $G$ is a connected graph with $\Delta(G)\le 3$, then for which graphs $i(G)=3\rho (G)$ hold?
	\end{question}
	
	In this paper, we solve Question \ref{q}, as detailed in the next theorem.
	\begin{theorem}\label{th1}
		If $G$ is a connected subcubic graph, then $i(G)=3\rho (G)$ if and only if $G\in\{H_1,H_2,H_3,H_4\}$ {\rm{(see Figure \ref{fig_1})}}.
	\end{theorem}
	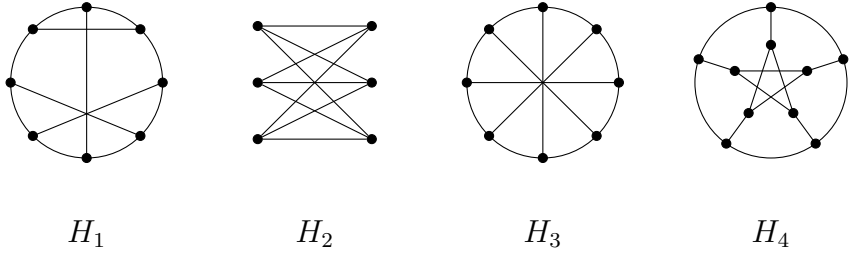
\begin{figure}[h]
		\centering
		\begin{tikzpicture}[scale=.5]
		\begin{scope}[xshift=-1cm]
		\coordinate (A) at (0,0);
		\foreach \x in {1,...,8}
		{
			\coordinate (P\x) at (\x*45:2cm);
			\node[draw,circle,inner sep=1.2pt,fill] at (P\x){};
		}
		\foreach \i/\j in {1/3,2/6,4/7,5/8}{
			\draw (P\i)--(P\j);
		}
		\draw (0,0) circle (2cm);
		\node at (0,-4) {$H_1$};
		
		\end{scope}
		\begin{scope}[xshift=4.5cm]
		\node[draw,circle,inner sep=1.2pt,fill](P1) at (-1,-1.5){};
		\node[draw,circle,inner sep=1.2pt,fill](P2) at (-1,0){};
		\node[draw,circle,inner sep=1.2pt,fill](P3) at (-1,1.5){};
		
		\node[draw,circle,inner sep=1.2pt,fill](Q1) at (2,-1.5){};
		\node[draw,circle,inner sep=1.2pt,fill](Q2) at (2,0){};
		\node[draw,circle,inner sep=1.2pt,fill](Q3) at (2,1.5){};
		\foreach \i in {1,2,3}{
			\draw (P1)--(Q\i);
			\draw (P2)--(Q\i);
			\draw (P3)--(Q\i);
		}

		\node at (0.5,-4) {$H_2$};
		
		\end{scope}
		\begin{scope}[xshift=11cm]
		\coordinate (A) at (0,0);
		\foreach \x in {1,...,8}
		{
			\coordinate (P\x) at (\x*45:2cm);
			\node[draw,circle,inner sep=1.2pt,fill] at (P\x){};}
		\foreach \i/\j in {1/5,2/6,3/7,4/8}{
			\draw (P\i)--(P\j);
		}
		\draw (0,0) circle (2cm);
		\node at (0,-4) {$H_3$};
		\end{scope}
		\begin{scope}[xshift=17cm]
		\coordinate (A) at (0,0);
		\foreach \x in {1,...,5}
		{
			\coordinate (P\x) at (\x*72+18:2cm);
			\node[draw,circle,inner sep=1.2pt,fill] at (P\x){};
			\coordinate (Q\x) at (\x*72+18:1cm);
			\node[draw,circle,inner sep=1.2pt,fill] at (Q\x){};				
		}
		\foreach \i/\j in {1/3,2/4,3/5,4/1,5/2}{
			\draw (P\i)--(Q\i);
			\draw (Q\i)--(Q\j);
		}
		\draw (0,0) circle (2cm);
		\node at (0,-4) {$H_4$};
		\end{scope}
		
		\end{tikzpicture}
		\caption{Tight graphs for Theorem \ref{th1}.}
		\label{fig_1}
	\end{figure}
	
	
	Our paper is organized as follows. In section 2, we present three propositions that are integral to the proof of Theorem \ref{th1}. Section 3 is dedicated to the proof of Theorem \ref{packingtheorem}. 
	In the last section, we provide the detailed proof of Theorem \ref{th1}.

	\section{Three Propositions}
	In this section, we will establish three crucial propositions that serve as fundamental tools in our primary proof of Theorem \ref{th1}. Before proving these propositions, it's essential to provide a brief restatement of the proof of Theorem \ref{th0} here.
	
	Let $G=(V(G),E(G))$ be a connected subcubic graph and let $S$ be a maximal packing of $G$.
	Define $N = N_G(S)$, $R = V(G) \setminus N_G[S]$ and $H = G[N]$. Since each vertex in $N$ has a neighbor in $S$ and the degree of each vertex in $G$ is at most 3, the induced subgraph $H$ has a maximum degree of at most 2. This implies that $H$ is the disjoint union of cycles, paths and isolated vertices. We may regard the set of paths of length 1 in $H$ as an induced matching $M$.	Let $W$ be the set of endpoints of the edges in $M$. Given $B\subseteq N$, let $X(B)$ be the set of all vertices $s\in S$ such that $|N_G(s)|=3$ and $N_G(s)\subseteq (N\setminus B)\cap W$. Now choose a set $A$ with $A \subseteq N$ that satisfies the following conditions.
	\begin{enumerate}[label={(\roman*)}]
		\item For each path $P_i$ of length at least $2$, $A$ contains the two endpoints of $P_i$. \label{i}
		\item \label{ii} $A$ is a maximal independent set in $H$ satisfying  \ref{i}.
		\item \label{iii} $|X(A)|$ is minimum among all choices satisfying \ref{i} and \ref{ii}.
	\end{enumerate}
	
	Based on the selection of $A$, $A$ is an independent dominating set of $H$. Let $T$ denote the set of vertices in $R$ that are not dominated by $A$. Given a vertex set $A$ that satisfies the conditions \ref{i}--\ref{iii}, Cho and Kim \cite{cho2023independent} presented a strategy for choosing an independent domination set $\hat{A}$ in $G$ that comprises of the following three sets:
	\begin{enumerate}
		\item[(I)] The vertex set $A$ satisfying \ref{i}-\ref{iii}.\label{1}
		\item[(II)] The vertex set, say $S^A$, of all vertices $s\in S$ that are not dominated by $A$.
		\item[(III)] An independent dominating set, say $Z$, of $G[T]$. 
	\end{enumerate}


	For convenience, let $S=S_1\cup S_2\cup S_3$, where $S_i$ denotes the set of all vertices in $S$ of degree $i$ for $1\le i\le 3$. Similarly, $S^A$ in (b) can be partitioned into $S^A_1\cup S^A_2\cup S^A_3$. Note by definition, $S^A_i\subseteq S_i$ for $1\le i\le 3$.
	Therefore, the authors \cite{cho2023independent} proved that for any vertex $s$ in $S^A_3$, there exists a neighbor, say $s^*$, of $s$ with two neighbors in $N$. For each $r\in Z$, there must be a neighbor, say $r^*$, of $r$ that is in $N$. Furthermore, they defined a one-to-one function $f$ from $\hat{A}=A\cup S^A\cup Z$ to $N\cup S^A_1\cup S^A_2$ as follows.

	\begin{equation}\label{eq1}
		f(v)=\begin{cases}
			v, & v\in A\cup S^A_1\cup S^A_2,  \\
			v^*, & v\in S^A_3 \cup Z. \\
		\end{cases}
	\end{equation}
	It implies that
	\begin{equation}\label{eq2}
		|\hat{A}|\le |N|+|S^A_1|+|S^A_2|\le |S_1|+2|S_2|+3|S_3|+|S^A_1|+|S^A_2|\le 3|S|.
	\end{equation}
	%
	The second inequality is derived from the fact that $|N| = |S_1| + 2|S_2| + 3|S_3|$, and the third inequality is based on the observation that $S = S_1 \cup S_2 \cup S_3$ and $S^A_i\subseteq S_i$ for $1\le i\le 3$. This completes the proof of Theorem \ref{th0}.
	
	Since the function defined in {\rm{(\ref{eq1})}} is of significance, we may restate that this function possesses the following properties.
	\begin{property}{\upshape\cite{cho2023independent}}\label{property1}
		Let $f$ be a function as defined in (\ref{eq1}). Then the vertices in $\hat{A}$ have the following properties:
		
		{\rm(A)} If $v \in A \cup S^A_1 \cup S^A_2$, then $f(v) = v$.
		
		{\rm(B)} If $v \in S^A_3 \cup Z$, then $v^* = f(v)$ is a neighbor of $v$. In particular, if $v \in S^A_3$, then $v^*$ has two neighbors in $N$.
	\end{property}
	
	With the help of Function (\ref{eq1}), we can deduce three propositions. In this section, the definitions of the symbols $A$, $S_i$ ($i\in[3]$), $S^A_i$ ($i\in[3]$), $N, H, M, W, Z, \hat{A}$ are the same with those defined in Theorem \ref{th0}. Let $\overline{S^A_i}=S_i\setminus S^A_i$. Recall that $M$ represents the set of paths of length 1 in $H$ and $W$ denotes the set of vertices in $M$. We begin with a helpful lemma which will be frequently used in the following. In the remaining part of this section, let $G$ be a connected subcubic graph with $i(G)=3\rho(G)$.
	\begin{lemma}\label{pro1cl_1}
		Let $S$ be any maximal packing of $G$. If $A$ is a set that satisfies conditions $\ref{i}$-$\ref{iii}$ and $\hat{A}$ is the corresponding independent dominating set as defined above, then the following statements hold.
		
		{\rm(a)} $S_1=\emptyset$ and $S_2=S^A_2$.
		
		{\rm(b)}  The function $f$ defined in (\ref{eq1}) is a bijection from $\hat{A}=A\cup S^A_2\cup S^A_3\cup Z$ to $N\cup S^A_2$.
		
		{\rm(c)} If a vertex $x\in N\setminus A$ and $d_H(x)=1$, then $x\in W$. 
	\end{lemma}
	\begin{proof}
		(a) By definition, $|S|\le \rho(G)$ and $i(G)\le |\hat{A}|$. If $i(G)=3\rho (G)$, then by inequality (\ref{eq2}), $|\hat{A}|=3|S|$. Hence, we obtain the following equation:
		\begin{equation}\label{eq2*}
			|N|+|S^A_1|+|S^A_2|=|S_1|+2|S_2|+3|S_3|+|S^A_1|+|S^A_2|= 3(|S_1|+|S_2|+|S_3|)
		\end{equation}
		With the help this equation, we deduce $|S^A_1| + |S^A_2| = 2|S_1| + |S_2|$. Since $S^A_i\subseteq S_i$ for $i\in [3]$, it follows that $|S^A_i|\le |S_i|$ holds, thus (a) is true.
		
		(b) Because $f$ is a one-to-one function from $\hat{A}$ to $N\cup S^A_2$, and equation $(\ref{eq2*})$ implies that $|\hat{A}| = |N| + |S^A_2|$, it follows that $f$ is a bijection.
		
		(c) Since $d_H(x)=1$, $x$ has a neighbor $x'$ in $H$. If $xx' \notin M$, then $x$ would be an endpoint of a path of length at least 2 in $H$, thus $x\in A$ by condition \ref{i}, this is a contradiction. Hence, $xx'\in M$, it follows from the definition of $M$ that $x \in W$.
	\end{proof}
	For convenience, let $f^{-1}$ be the inverse function of $f$.
	Now we are ready to prove the following three propositions.
	\begin{proposition}\label{cubic}
		The graph $G$ is cubic and every maximal packing of $G$ is maximum.    
	\end{proposition}
	\begin{proof}
		We will establish three claims and an algorithm to prove the conclusion. In the following proof, let $A$ be the set that satisfies conditions \ref{i}--\ref{iii}.
		
		\begin{claim}\label{pro1_cl2}
			There is no vertex of degree $1$ in $G$.
		\end{claim}
		\begin{proof}
			Assume $x$ is a vertex of degree $1$ in $G$. There exists a maximal packing $S$ containing $x$. Thus, $x\in S_1$, which contradicts Lemma \ref{pro1cl_1} (a).
		\end{proof}
		
		Next, we demonstrate that $G$ contains no vertices of degree 2. Assume $y$ is a vertex with $d_G(y) = 2$. Let $S$ be a maximal packing containing $y$, and let $\hat{A} = A \cup S^A_2 \cup S^A_3 \cup Z$ be the corresponding independent dominating set. By Lemma \ref{pro1cl_1} (a), $S_2 = S^A_2$, so $y \in S^A_2$.
		Let $y_1$ and $y_2$ be the two neighbors of $y$ in $G$. Then we obtain the following two claims:
		
		\begin{claim}\label{pro1_cl3}
			For each $i\in [2]$, $y_i\in W\setminus A$.
		\end{claim}
		\begin{proof}
			By the definition of $S^A_2$, it follows that $y_i \notin A$. Since $y_i \in N$ and $f$ is a bijection from $\hat{A} = A \cup S^A_2 \cup S^A_3 \cup Z$ to $N \cup S^A_2$, there exists a vertex $y_i' \in \hat{A}$ such that $f(y_i') = y_i$. If $y_i' \in A$ then $y_i = y_i'$ by Property \ref{property1} (A), a contradiction. If $y_i' \in S^A_2$ then $y_i = y_i'$ by Property \ref{property1} (A), but $y_i$ cannot belong to both $N$ and $S$. If $y_i' \in S^A_3$ then $y_iy_i' \in E(G)$ according to Property \ref{property1} (B), which means that $y_i$ is a common neighbor of $y$ and $y_i'$, but both $y$ and $y_i'$ belong to $S$, this is impossible as $S$ is a packing. Therefore, $y_i' \in Z$ and $y_iy_i' \in E(G)$ according to Property \ref{property1} (B). Now, $y_i$ has $\{ y, y_i'\}$ as its neighbors, since $\Delta(G) \le 3$, either $d_H(y_i) = 0$ or $d_H(y_i) = 1$. If $d_H(y_i) = 0$, by condition \ref{ii} of $A$, we obtain $y_i \in A$, which is impossible. Therefore, $d_H(y_i) = 1$, it follows from Lemma \ref{pro1cl_1} (c) that $y_i\in W$.
		\end{proof}
		Since $S_1=\emptyset$ and $S_2=S^A_2$, we observe that $A\subseteq N(\overline{S^A_3})$. According to the condition \ref{ii} of $A$, $A$ is a dominating set of $H$. Therefore, by Claim \ref{pro1_cl3}, $y_1$ has a neighbor, say $s_{11}$, in $ A\cap W$.
		Let $s_1\in \overline{S^A_3}$ be a neighbor of $s_{11}$ with $N(s_1)=\{s_{11},s_{12},s_{13}\}$. In the next claim, we show that $s_{12}$ and $s_{13}$ are contained in $W$.  
		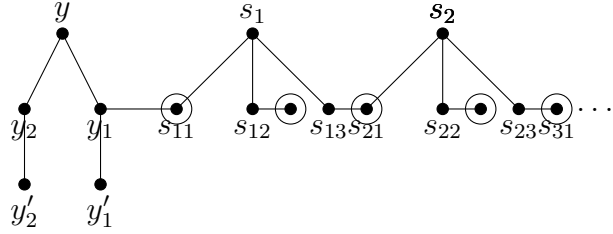
\begin{figure}[h]
			\centering
			\begin{tikzpicture}
			\foreach \i in {0.5,1.5,2.5,3.5,4,4.5,5,6,6.5,7,7.5}
			{
				\node[draw,circle,inner sep=1.5pt,fill] at (\i,0){};
			}
			\node[draw,circle,inner sep=1.5pt,fill] at (1,01){};
			\node[draw,circle,inner sep=1.5pt,fill] at (3.5,01){};
			\node[draw,circle,inner sep=1.5pt,fill] at (0.5,-01){};
			\node[draw,circle,inner sep=1.5pt,fill] at (1.5,-01){};
			\draw (0.5,0)node[anchor=north]{$y_2$}--(1,1)node[above]{$y$}--(1.5,0)node[anchor=north]{$y_1$}--(2.5,0)node[anchor=north]{$s_{11}$}--(3.5,1)node[above]{$s_1$}--(3.5,0)node[anchor=north]{$s_{12}$};
			
			\draw (3.5,0)--(4,0);
			\draw (3.5,1)--(4.5,0)node[anchor=north]{$s_{13}$}--(5,0)node[anchor=north]{$s_{21}$};
			\draw (1.5,0)--(1.5,-1)node[below]{$y_1'$};
			\draw (.5,0)--(.5,-1)node[below]{$y_2'$};
			\node[draw,circle,inner sep=1.5pt,fill] at (6,1){};
			\foreach \j in {5,6,7}
			{
				\draw (6,1)node[above]{$s_{2}$}--(\j,0);
			}
			\draw (6,0)node[below]{$s_{22}$}--(6.5,0);
			\draw (7,0)node[below]{$s_{23}$}--(7.5,0)node[below]{$s_{31}$};
			\node at (8,0) {$\ldots$};
			\draw (2.5,0) circle (.2);
			\draw (4,0) circle (.2);
			\draw (5,0) circle (.2);
			\draw (6.5,0) circle (.2);
			\draw (7.5,0) circle (.2);
			\end{tikzpicture}
			\caption{The induced subgraph $G[H_1]$ for Proposition \ref{cubic}, where the circled vertices belong to $A$.}
			\label{fig:cubicproof}
		\end{figure}
		
		\begin{claim}\label{pro1_cl4}
			For $i\in \{2,3\}$, $s_{1i}\in W\setminus A$.
		\end{claim}
		\begin{proof}
			If $s_{1i}\in A$ for some $i\in \{2,3\}$, then we can construct a new set $A'=A\setminus \{s_{11}\} \cup \{y_1\}$. Obviously, $A'$ satisfies conditions \ref{i}--\ref{iii}. Hence, $y\in S_2\setminus S^{A'}_2$, which implies $S_2\neq S^{A'}_2$, this contradicts Lemma \ref{pro1cl_1} (a). Note that $\Delta (H) \le 2$. If $d_H(s_{1i})=0$, by the condition \ref{ii} of $A$, $s_{1i}\in A$, which is impossible. If $d_H(s_{1i})=2$, we consider the preimage $f^{-1}(s_{1i})$. According to Lemma \ref{pro1cl_1} (b), $f^{-1}(s_{1i}) \in A\cup S^A_2\cup S^A_3\cup Z$.
			
			Since $s_{1i}\in N\setminus A$, by Property \ref{property1} (A), $f^{-1}(s_{1i})\notin A\cup S^A_2$.
			If $d_H(s_{1i})=2$, then $s_{1i}$ has no neighbor in $S^A_3\cup Z$, hence, $f^{-1}(s_{1i})\notin S^A_3\cup Z$ by Property \ref{property1} (B). Therefore, $f^{-1}(s_{1i})$ doesn't exist, which is impossible since $f$ is a bijection. Thus, we conclude that $d_H({s_{1i}})=1$, it follows from Lemma \ref{pro1cl_1} (c) that $s_{1i}\in W$.
		\end{proof}
		We now present an algorithm to generate a subgraph of $G$, as illustrated in Figure \ref{fig:cubicproof}. 
		\begin{algorithm}\label{alg}  
			\caption{A generation of $G[H_1]$}  
			\begin{algorithmic}[1]  
				
				\State Let $m=1$ and $H_1=N[y]$
				\While{$s_{m2}\in W\setminus A$ and $s_{m3}\in W\setminus A$}
				\State $H_1=H_1\cup N[s_m]$		
				\State $m=m+1$
				\State Set $s_{m1}$ be the neighbor of $s_{(m-1)3}$ in $H$
				\State Set $N(s_m)=\{s_{m1},s_{m2},s_{m3}\}$ 			 
				\EndWhile  
				\State \textbf{return} $G[H_1]$
			\end{algorithmic}  
		\end{algorithm}
		
		Note that Algorithm 1 is valid because Claim \ref{pro1_cl4} guarantees that the neighbors of $s_1$ satisfy the initial conditions of the algorithm. We claim that Algorithm 1 should terminate at some vertex $s_t$ $(t\geq 2)$. To prove this, we need to show that the vertices in the algorithm are different from each other, which means that the algorithm will not return to a previous vertex. Let $Q=\{s_{mj}:m\ge 1~ \text{and}~j\in [3]\}$. First, we show that all vertices in $Q$ are different. Let $k_1$ and $k_2$ be two positive integers and let $j_1,j_2\in [3]$. If $s_{k_1j_1}\in Q$ and $s_{k_2j_2}\in Q$ are the same vertex, then $d(s_{k_1},s_{k_2})\le 2$, a contradiction to that $S$ is a packing. Next, we show that any two vertices of $Q$ have no common neighbor in $H$. Note that $s_{11}\in A\cap W$, and by Claim \ref{pro1_cl4}, we have $s_{1i}\in W\setminus A$ for each $i\in \{2,3\}$. Moreover, for $m\geq 2$, we have $s_{m1}\in A\cap W$ according to line 5 of Algorithm 1, and $\{s_{m2},s_{m3}\}\subseteq W\setminus A$ according to line 2 in Algorithm 1. 
		This leads to the conclusion that $Q$ is a subset of $W$. Consequently, any two vertices of $Q$ have no common neighbors in $H$. 
		Therefore Algorithm 1 must terminate at some vertex $s_t$ since $G$ is a finite graph. If both $s_{t2}$ and $s_{t3}$ are outside of $A$, then by a similar argument of Claim \ref{pro1_cl4}, we can continue the Algorithm, which contradicts the assumption. Therefore, there exists some $j \in \{2,3\}$ such that $\{s_{t1}, s_{tj}\} \subseteq A$.
		
		Let $A_{new}=(A\setminus \{s_{11},s_{21},\ldots, s_{t1}\}) \cup \{y_1,s_{13},s_{23} \ldots,s_{(t-1)3}\}$. 
		Then $A_{new}$ is a vertex set satisfying the conditions \ref{i}--\ref{iii}, (this is true, because $s_t$ is dominated by $s_{tj}$) hence, $y\in S_2\setminus S^{A_{new}}_2$, a contradiction to Lemma \ref{pro1cl_1} (a). Hence, $G$ is $3$-regular. Hence, $f$ is a bijection from $\hat{A}=A\cap S^A_3\cap Z$ to $N$. It follows that $|\hat{A}|=|N|=3|S|$.

		Furthermore, each maximal packing is also a maximum packing, as $i(G)=3\rho (G)$ and $3\rho (G)\ge 3|S|= |\hat{A}|\ge i(G)$ for every maximal packing $S$.
	\end{proof}
	
	Next, we study the structure of the induced graph $G[N(S)]$. Before proving Proposition \ref{H}, we present the following lemma.
	
	\begin{lemma}\label{pro2}
		Let $S$ be a maximal packing of $G$. Then there exists a vertex set $A$ that satisfies conditions $\ref{i}$-$\ref{iii}$ such that $S^A_3=\emptyset$.
		
	\end{lemma}
	
	\begin{proof}
		Let $|S^A_3|$ be the minimum among all choices of $A$ satisfying conditions (i)-(iii). 
		We now prove that $S^A_3=\emptyset$.
		By contradiction, assume $s_1\in S^A_3$. Let $N(s_1)=\{s_{11},s_{12},s_{13}\}$ with $d_H(s_{11})\geq d_H(s_{12})\geq d_H(s_{13})$. Then we have a claim as follows. 
		\begin{claim}\label{pro2_cla1}			
			The degrees of $s_{11},s_{12},s_{13}$ are $2$, $1$, $1$ in $H$ respectively. In particular, $s_{12}$ and $s_{13}$ belong to $W\setminus A$.
		\end{claim}
		
		\begin{proof}
			Note that, according to the definition of $S^A_3$ no vertex in $N(s_1)$ belongs to $A$.
			By condition \ref{ii} of $A$ and $\Delta(G)=3$,  
			$1\leq d_H(s_{1i})\leq 2$ for each $i \in [3]$.
			Moreover, since $f$ is a bijection and by Property \ref{property1} (B), $u$ has exactly one neighbor of degree 2 in $H$, so $d_H(s_{11}) = 2$ and $d_H(s_{12}) =d_H(s_{13})=1$. 
			Therefore, by Lemma \ref{pro1cl_1} (c), $u_2$ and $u_3$ belong to $W\setminus A$.
		\end{proof}

		By Claim \ref{pro2_cla1}, $s_{13}\in W\setminus A$. Since $A$ is a dominating set of $H$, $s_{13}$ has a neighbor, say $s_{21}$, that belongs to $A$. Let $s_2\in \overline{S^A_3}$ with $N(s_2)=\{s_{21},s_{22},s_{23}\}$. We claim that $s_{2i}\in W\setminus A$ for $i\in \{2,3\}$. If $s_{2i}\in A$, then we construct a set $A'=(A\setminus \{s_{21}\})\cup \{s_{13}\}$. It's easy to verify that $A'$ satisfies conditions \ref{i}--\ref{iii}, but $s_1\notin S^{A'}_3$ and $|S^{A'}_3|<|S^{A}_3|$, a contradiction. Hence, neither $s_{22}$ nor $s_{23}$ is in $A$. Next, we show that $s_{2i}\in W$. Note that $\Delta (H) \le 2$.
		If $d_H(s_{2i})=0$, by the condition \ref{ii} of $A$, $s_{2i}\in A$, which is impossible. If $d_H(s_{2i})=2$, we consider the preimage $f^{-1}(s_{2i})$ of $s_{2i}$. By Lemma \ref{pro1cl_1} (b) and Proposition \ref{cubic}, $f^{-1}(s_{2i}) \in A\cup S^A_3\cup Z$.   
		Since $s_{2i}\in N\setminus A$, by Property \ref{property1} (A), $f^{-1}(s_{1i})\notin A$.
		Since $s_{2i}$ has no neighbor in $S^A_3\cup Z$, $f^{-1}(s_{2i})\notin S^A_3\cup Z$ by Property \ref{property1} (B). Therefore, $f^{-1}(s_{2i})$ doesn't exist, this is impossible because $f$ is a bijection. Thus, we conclude that $d_H({s_{2i}})=1$, by Lemma \ref{pro1cl_1} (c), $s_{2i}\in W$.
		\begin{figure}[h]
			\centering
			\begin{tikzpicture}
			\foreach \i in {2,2.5,3,3.5,4,4.5,5,6,6.5,7,7.5}
			{
				\node[draw,circle,inner sep=1.5pt,fill] at (\i,0){};
			}
			\draw (1.8,0)--(3.2,0);
			\node[draw,circle,inner sep=1.5pt,fill] at (3.5,01){};
			
			\draw (2.5,0)node[anchor=north]{$s_{11}$}--(3.5,1)node[above]{$s_1$}--(3.5,0)node[anchor=north]{$s_{12}$};
			
			\draw (3.5,0)--(4,0);
			
			\draw (3.5,1)--(4.5,0)node[anchor=north]{$s_{13}$}--(5,0)node[anchor=north]{$s_{21}$};
			
			\node[draw,circle,inner sep=1.5pt,fill] at (6,1){};
			\foreach \j in {5,6,7}
			{
				\draw (6,1)node[above]{$s_{2}$}--(\j,0);
			}
			\draw (6,0)node[below]{$s_{22}$}--(6.5,0);
			\draw (7,0)node[below]{$s_{23}$}--(7.5,0)node[below]{$s_{31}$};
			\node at (8,0) {$\ldots$};
			
			\draw (4,0) circle (.2);
			\draw (5,0) circle (.2);
			\draw (6.5,0) circle (.2);
			\draw (7.5,0) circle (.2);
			\end{tikzpicture}
			\caption{The induced subgraph $G[H_2]$ for Lemma \ref{pro2}, where the circled vertices belong to $A$.}
			\label{fig:cubicproof2}
		\end{figure}
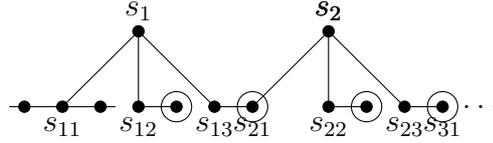

		We now present an algorithm to generate the subgraph $G[H_2]$ of $G$ (see Figure \ref{fig:cubicproof2}):
		\begin{algorithm}[H]\label{alg2}  
			\caption{A generation of $G[H_2]$}  
			\begin{algorithmic}[1]  
				
				\State Let $m=2$ and $H_1=N[s_1]$
				\While{$s_{m2}\in W\setminus A$ and $s_{m3}\in W\setminus A$}
				\State $H_1=H_1\cup N[s_m]$		
				\State $m=m+1$
				\State Set $s_{m1}$ be the neighbor of $s_{(m-1)3}$ in $H$
				\State Set $N(s_m)=\{s_{m1},s_{m2},s_{m3}\}$ 			 
				\EndWhile  
				\State \textbf{return} $G[H_2]$
			\end{algorithmic}  
		\end{algorithm}
		By a similar argument as in the proof of Proposition \ref{pro2}, we can establish that this algorithm must terminate at some vertex $s_k$ where there exists an index $j\in \{2,3\}$ such that $\{s_{k1}, s_{kj}\} \subseteq A$. We define $A_{new}$ as $(A\setminus \{s_{21}, \ldots, s_{k1}\}) \cup \{s_{13}, s_{23}, \ldots, s_{(k-1)3}\}$. It can be verified that $A_{new}$ satisfies conditions \ref{i}--\ref{iii}. However, it follows that $s_1$ is not a member of $S^{A_{new}}_3$, which contradicts the minimality of $|S^A_3|$.
	\end{proof}		
	\begin{proposition}\label{H}
		Let $S$ be a maximal packing of $G$. Then $G[N(S)]$ consists of matching edges and isolated vertices.     
	\end{proposition}
	
	\begin{proof}
		By Lemma \ref{pro2}, let $A$ be a set that satisfies the conditions \ref{i}--\ref{iii} and $S^A_3=\emptyset$. 
		As we have defined earlier, $H=G[N(S)]$ and $N=N(S)$. It follows from Propositions \ref{cubic} and Lemma \ref{pro2} that $\hat{A}=A\cup Z$ and $|\hat{A}|=|N|$.
		For any vertex $x\in N$, if $d_H(x)=2$, it means $x$ has no neighbor in $Z$. Consequently, $g^{-1}(x)\notin Z$, leading to the conclusion that $g^{-1}(x)\in A$. By Property \ref{property1}, we can deduce that $x\in A$. If there is a cycle or a path $P_k$ with $k\ge 4$ in $H$, then there are two adjacent vertices of degree 2, and they all belong to $A$, which is impossible since $A$ is an independent set. 
		If there is a $P_3$ in $H$, the endpoints of the path would belong to $A$ by condition \ref{i} of $A$. This implies that $V(P_3)\subseteq A$, which is impossible. Hence, $\Delta (H)\le 1$, as desired.
	\end{proof}
	
	Recall that $R=V(G)\setminus N[S]$, let $\overline{Z}=R\setminus Z$ and $Z^*=\{f(v) |v\in Z\}$. Then we obtain the following proposition. 
	
	\begin{proposition}\label{prop3}
		Let $S$ be a maximal packing of $G$.
		If $A$ is a set that satisfies the conditions $\ref{i}$-$\ref{iii}$ and $S^A_3=\emptyset$, then we have the following facts.
		
		$(1)$ $Z^*$ and $A$ is a partition of $N$.
		
		$(2)$ Every vertex $z\in Z$ has a neighbor $z^*\in Z^*$, and two neighbors in $\overline{Z}$.
		
		$(3)$ Every vertex in $A$ has at least one neighbor in $\overline{Z}$.
		
		$(4)$ If $uv\in E(H)$, then one of $\{u,v\}$ is in $A$, and another vertex is in $Z^*$.
	\end{proposition}
	\begin{proof}
		
		
		
		$(1)$ It is straightforward to deduce (1) via the bijection $f$ defined in Function (\ref{eq1}).
		
		$(2)$ It is evident from Property \ref{property1} (B) that each vertex in $Z$ has at least one neighbor in $Z^*$. Suppose there exists a vertex $z\in Z$ that has two neighbors, $z_1$ and $z_2$, in $Z^*$. Since $f$ is a bijection, without loss of generality, let $f(z)=z_1$. Then there exists a vertex $z'\in Z\setminus \{z\}$ such that $f(z')=z_2$. It follows from Proposition 1 that $\Delta(G)=3$, so $z_2$ is isolated in $H$. Thus $z_2\in A$ by condition (ii), which is a contradiction. Therefore, every vertex of $Z$ has exactly one neighbor in $Z^*$ and two neighbors in $\overline{Z}$.
		
		$(3)$ It is easy to deduce that (3) holds from Propositions \ref{cubic} and \ref{H}.
		
		$(4)$ By Proposition \ref{H} and the conditions of $A$, exactly one of the vertices $\{u,v\}$ belongs to $A$. Moreover, it follows from $(1)$ that the other vertex is in $Z^*$.
		
	\end{proof}

	\section{Proof of Theorem \ref{packingtheorem}}
	Before providing the proof of Theorem \ref{packingtheorem}, we would like to present a fundamental observation and introduce a lemma. By employing these, we will demonstrate Theorem \ref{packingtheorem} through the calculation of the clique number of the power graph associated with a given graph. The \textit{clique number}, in this context, refers to the maximum cardinality of complete subgraph for the graph.
	\begin{observation}\label{obser} {\upshape\cite{2017Packing}}
		If $G$ is a graph, then $\rho (G)=\alpha(G^2)$.
	\end{observation}
	
	\begin{lemma}{\upshape\cite{henning2014new}} \label{keylemma}
		If $G$ is a graph of order $n$ and $p$ is an integer, such that $(A)$ below holds, then $\alpha(G)\ge \frac{2n}{p}$.
		
		$(A)$: For every clique $X$ in $G$ there exists a vertex $x\in X$, such that $d(x)< p-|X|$.
	\end{lemma}
	
	\noindent\textbf{\textit{Proof of Theorem \ref{packingtheorem}.}}
	Assume, for the sake of contradiction, there exists a cubic graph $G$ with girth 5 and $n \geq 12$, such that $\rho(G) < \frac{n+1}{8}$.
	Now, we present an algorithm for constructing a packing, denoted as $S$, within the graph $G$.
	\begin{algorithm} [h]
		\label{alg3}
		\caption{A construction of $S$}  
		\begin{algorithmic}[1]  
			\State For any vertex $u\in V(G)$, set $S=\{u\},\overline{S}=V(G)\setminus \{u\}$   
			\While{$\overline{S}\neq \emptyset$}
			\State For any $v\in \partial_3(S)$, $S\leftarrow S\cup \{v\}$
			\State $\overline{S}\leftarrow V(G)\setminus N_2[S]$  
			\EndWhile  
			\State \textbf{return} $S$
		\end{algorithmic}  
	\end{algorithm}
	
	
	When the algorithm stops, the resulting set $S$ qualifies as a packing, as the distance between any pair of vertices in $S$ is at least 3. Since $g(G)=5$, for each $w \in V(G)$, every two vertices in $N_2[w]$ are distinct. Consequently, $N_2[w]= 10$. After selecting the initial vertex $u$, 10 vertices are eliminated. Subsequently, with each choice of vertex $v$, we eliminate at most 8 vertices, as at least one neighbor of $v$ and one vertex at a distance of 2 from $v$ have already been removed. Thus, we establish the following two claims.
	\begin{claim}\label{Adjacentclaim1}
		For each vertex $u_1$ of $G$ and $u_2$ of $\partial_3(u_1)$, $|N_2[u_1]\cap N_2[u_2]|\le 4$.
	\end{claim}

	\begin{proof}
		Suppose that there exists a vertex $u'_2 \in \partial_3(u_1)$ such that $|N_2[u_1]\cap N_2[u'_2]|\ge 5$.
		By Algorithm 3, let $u_1$ be the first selected vertex and $u'_2$ be the second, we obtain that $n\leq |N_2[u_1]|+(|N_2[u_2']\setminus N_2[u_1]|)+8(|S|-2)\leq 10+5+8(|S|-2)$, which implies $\rho(G) \ge \frac{n+1}{8}$, a contradiction.
	\end{proof}
	\begin{claim}\label{Adjacentclaim2}
		Any two distinct cycles of length $5$ in $G$ have at most one common edge.
	\end{claim}
	\begin{proof}
		Suppose there are two distinct cycles of length 5 sharing $k$ common edges. Clearly, $0 \leq k \leq 3$.
		
		Note $g(G)=5$. If $k=3$, then $G$ contains a $C_4$, a contradiction. If $k=2$ and the two edges are nonadjacent, then $G$ contains a $C_3$, a contradiction. Therefore, let $abcdea$ and $fgcdef$ be two cycles of order $5$ that share two adjacent edges, as illustrated in Figure \ref{fig:cubicgirthproof}. We deduce that $d(b,f)=2$. Otherwise, if $d(b,f)=1$, there is a $C_4=bcgfb$, a contradiction; if $d(b,f)=3$, then $|N_2[b]\cap N_2[f]|\ge 5$, contradicting Claim \ref{Adjacentclaim1}.
		
		By symmetry, we may assume that $d(a,g)=2$. Let $z_1=N(b)\cap N(f)$ and $z_2=N(a)\cap N(g)$. Since $\Delta(G)=3$, $z_1\neq z_2$. We next show that $\max \{d(d,z_1),d(d,z_2)\} =3$. If $dz_i\in E(G)$ for some $i\in [2]$, then there is a $C_4$, a contradiction. If $d(d,z_1)=d(d,z_2)=2$, then there is a common neighbor of $d,z_1,z_2$. Thus, $G$ is the Petersen graph, which contradicts that $n\geq 12$. Without loss of generality, let $d(d,z_1)=3$. Then $|N_2[d]\cap N_2[z_1]|\ge 6$, which contradicts Claim \ref{Adjacentclaim1}.
		\begin{figure}
			\centering
			\begin{tikzpicture}[scale=1.5]

			\node[draw,circle,inner sep=1.5pt,fill](A)at (0,2){};
			\node[draw,circle,inner sep=1.5pt,fill](B) at (-1,1.4){};
			
			\node[draw,circle,inner sep=1.5pt,fill](C) at (1,1.4){}; 
			
			\node[draw,circle,inner sep=1.5pt,fill](D) at (-1,.6){};
			
			\node[draw,circle,inner sep=1.5pt,fill](E) at (0,1.4){};
			
			\node[draw,circle,inner sep=1.5pt,fill](F) at (0,.6){};
			
			\node[draw,circle,inner sep=1.5pt,fill](G) at (1,.6){};
			
			\node[draw,circle,inner sep=1.5pt,fill](H) at (-.7,.2){};
			\node[draw,circle,inner sep=1.5pt,fill](I) at (.7,.2){};
			\draw (A)--(B)--(D)--(F)--(E)--(A)--(C)--(G)--(F);
			\draw (D)--(I);
			\draw (G)--(H);
			\draw (B) to(H);
			\draw (I) to(C);
			\node at (0,2.3) {$e$};
			\node at (1.2,1.4) {$f$};
			\node at (-1.2,.6) {$b$};
			\node at (-1.2,1.4) {$a$};
			\node at (1.2,.6) {$g$};
			\node at (.7,0) {$z_1$};
			\node at (-.7,0) {$z_2$};
			\node at (0.2,1.4) {$d$};
			\node at (0.2,.7) {$c$};
			\end{tikzpicture}
			\caption{Proof of Claim \ref{Adjacentclaim2}.}
			\label{fig:cubicgirthproof}
		\end{figure}
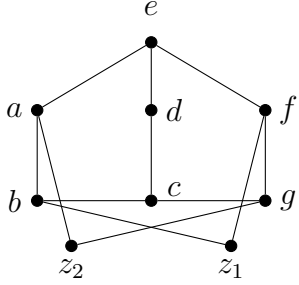
	\end{proof}
	We now consider the square, $G^2$, of $G$. Let $x$ be any vertex of $G$.
	Assume that the cardinality of every clique in $G^2$ containing $x$ is at most 5, and let $p=15$. 
	Take an arbitrary clique $X$ in $G^2$ containing the vertex $x$. 
	It is clear that $p-|X|\ge 15-5> 9=d_{G^2}(x)$. 
	Therefore, by Lemma \ref{keylemma} and the Observation \ref{obser}, we deduce that $\rho(G)=\alpha(G^2)\ge \frac{2n}{p}=\frac{2n}{15}$. Furthermore, since $\rho(G)$ is an integer, it implies that $\rho(G) \geq \lceil \frac{2n}{15} \rceil \ge  \frac{n+1}{8}$ for $n\ge 12$, a contradiction. 
	Thus, in the following proof, we are devoted to showing that $|X|\le 5$.
	
	For convenience, let $N(x)=\{x_1,x_2,x_3\}$, $N(x_1)=\{x,y_1,y_2\}$, $N(x_2)=\{x,y_3,y_4\}$ and $N(x_3)=\{x,y_5,y_6\}$. Observe that these vertices are distinct from each other since $g(G)=5$ and $V(X)\subseteq N_2[x]$. In the following proof in this section, we will use these notations without explanation. A straightforward but helpful fact can be easily proved.
	\begin{claim}\label{fact1}
		If $x_i\in V(X)$, then for each $j\in [3]\setminus \{i\}$, $N(x_j)\setminus \{x\}\nsubseteq V(X)$.
	\end{claim}
	\begin{proof}
		Without loss of generality, let $x_1\in V(X)$ and $N(x_2)\subseteq V(X)$, that is $\{y_3,y_4\}\subseteq V(X)$. 
		Therefore $d_G(y_3,x_1)=d_G(y_4,x_1)=2$. We may assume that $y_2$ is a common neighbor of $y_3$ and $x_1$. Then either $y_2$ or $y_1$ would be the common neighbor of $y_4$ and $x_1$. If $y_2y_4\in E(G)$, then $x_2y_3y_2y_4x_2$ forms a $C_4$, which is impossible. If $y_4y_1\in E(G)$, then $xx_1y_1y_4x_2x$ and $xx_1y_2y_3x_2x$ are two cycles of length 5 that share two common edges, which contradicts Claim \ref{Adjacentclaim2}.
	\end{proof}
	\begin{claim}\label{Thm3_claim4}
		The clique number of $G^2$ is at most 5.
	\end{claim}
	
	\begin{proof}
		Assume that $X$ is a clique of order 6 in $G^2$ containing $x$, where $x$ is an arbitrary vertex of $V(G)$. Let $Y=\{y_i: i\in[6]\}$. We consider four cases based on the number of vertices of $N(x)\cap V(X)$.
		
		\textbf{Case 1.} $N(x)\cap V(X)=\emptyset$. 
		
		Obviously, $|V(X)\cap Y| = 5$, without loss of generality, we may assume that $\{y_i: i\in[5]\}\subseteq V(X)$.
		Because $X$ is a clique in $G^2$, the distance between $y_5$ and any other vertex $y_l$ is at most 2 for $l \in [4]$. 
		
		Now, we prove that $y_5$ has no neighbor in $Y$. Clearly, $y_5y_6\notin E(G)$ as $g(G)\geq 5$. Suppose that $y_i\in Y\setminus \{y_5,y_6\}$ is a neighbor of $y_5$ in $G$.
		Without loss of generality, let $y_4y_5\in E(G)$. Then by $g(G)=5$ and Claim \ref{Adjacentclaim2}, $y_4$ has no neighbor in $\{y_1,y_2,y_3\}$. As $\Delta (G)=3$, $x_3$ also has no neighbor in  $\{y_1,y_2,y_3\}$. Next, we consider the third neighbor $z$ of $y_5$ that is distinct from $\{y_4,x_3\}$. If $z \in \{y_1,y_2,y_3\}$, then $z$ is adjacent to each vertex in $\{y_1,y_2,y_3\} \setminus \{z\}$ since $d_G(y_5,y_i) \le 2$ for each $i \in [3]$.
		This leads to $d_G(z) \ge 4$, which is a contradiction. Therefore, $z \notin Y$, which means $d_G(y_5,y_i) = 2$ for each $i \in [3]$. Consequently, $z$ would be the common neighbor of $y_i$ and $y_5$ for each $i \in [3]$. This implies $d_G(z) \ge 4$, which is not possible. 
		
		Hence, we can deduce that the distance between $y_5$ and $y_l$ is exactly 2 for each $l \in [4]$.
		Then the two neighbors of $y_5$ excluding $x_3$, say $y_{51}$ and $y_{52}$, satisfy the condition that $N(y_{51})\cup N(y_{52})=\{y_i: i\in[5]\}$. 
		In this case, we can deduce that $|N_2[y_{51}]\cap N_2[x]|\ge 6$, contradicting Claim \ref{Adjacentclaim1}. 
		
		\textbf{Case 2.}  $|N(x)\cap V(X)|=1$.
		
		By symmetry, assume $x_1\in V(X)$, then $|V(X)\cap Y|=4$. According to Claim \ref{fact1} and $|X|=6$, we have $|\{y_3,y_4\}\cap V(X)|=1$, $|\{y_5,y_6\}\cap V(X)|=1$ and $\{y_1,y_2\}\subseteq V(X)$. Without loss of generality, we may assume $\{y_3,y_5\}\subseteq V(X)$. 
		Since $d(y_3,x_1)=d(y_5,x_1)=2$, $y_j$ has a neighbor in $\{y_1,y_2\}$ for each $j\in \{3,5\}$. 
		If $y_3$ and $y_5$ have a common neighbor, say $y_c$, in $\{y_1,y_2\}$, then $xx_1y_cy_3x_2x$ and $xx_1y_cy_5x_3x$ are two cycles of length 5, which contradicts Claim \ref{Adjacentclaim2}.
		By symmetry, we can assume $y_3y_2\in E(G)$, then $y_5y_1\in E(G)$. 
		It follows from Claim \ref{Adjacentclaim2} that $y_5$ is not adjacent to any vertex in $\{y_2,y_3,y_4\}$. Furthermore, since $g(G)=5$, $y_2$ and $y_5$ do not have a common neighbor in $Y$.
		Therefore, $y_2$ and $y_5$ have a common neighbor, say $z$, outside $N_2[x]$ since $d(y_2,y_5)=2$.
		Hence $d(x,z)=3$ and $|N_2[x]\cap N_2[z]|\ge 6$, contradicting Claim \ref{Adjacentclaim1}.
		
		\textbf{Case 3.} $|N(x)\cap V(X)|=2$.
		
		We may assume $\{x_1,x_2\}\subseteq V(X)$. 
		By Claim \ref{fact1} and $|X|=6$, we have $|V(X)\cap N(x_i)|=1$ for each $i\in [3]$. 
		Assume $y_k\in V(X)\cap N(x_3)$, since $d(x_1,y_k)=d(x_2,y_k)=2$, it follows that $y_k$ has a neighbor in $\{y_1,y_2\}$ as well as a neighbor in $\{y_3,y_4\}$. 
		Consequently, two cycles of length 5 share two common edges, $xx_3$ and $x_3y_k$, which contradicts Claim \ref{Adjacentclaim2}.
		
		\textbf{Case 4.} $|N(x)\cap V(X)|=3$.
		
		Note that there exists a vertex $y_i\in V(X)$, where $i\in [6]$, without loss of generality, let it be $y_1$. Then $d(y_1,x_2)=d(y_1,x_3)=2$, which implies that $y_1$ has a neighbor in $\{y_3,y_4\}$ as well as a neighbor in $\{y_5,y_6\}$. 
		As a result, we can find two cycles of length 5 that share two edges, $xx_1$ and $x_1y_1$, which is a contradiction with Claim \ref{Adjacentclaim2}.
		
		Due to $x$ is arbitrary, we can establish the claim.
	\end{proof} 
	By the claims established above, we can deduce the desired result.
	\qed
	
	\section{Proof of Theorem \ref{th1}}
	
	Before proving Theorem \ref{th1}, we first introduce a lemma proposed by Henning et al., as shown below. By combining this lemma with three Propositions and Theorem \ref{packingtheorem}, we can complete the proof of Theorem \ref{th1}.
	
	\begin{lemma}{\upshape\cite{dorbec2015independent}}\label{Henningconj}
		If $G\neq C_5\Box K_2$ is a connected cubic graph of order $n$ that doesn't have a subgraph isomorphic to $K_{2,3}$ , then $i(G)\le \frac{3n}{8}$.  
	\end{lemma} 
	\noindent\textbf{\textit{Proof of Theorem \ref{th1}.}}
	If $G \in \{H_1,H_2,H_3,H_4\}$, then $i(G)=3$ and $\rho(G)=1$, which shows the sufficiency of the Theorem. 
	
	Now, let us prove the necessity. By Proposition \ref{cubic}, we may assume that $G$ is a connected cubic graph with $i(G)=3\rho(G)$. In the next claim, we show that $g(G)\le 5$, and then obtain the extremal graphs by discussing the girth in different cases.
	\begin{claim}
		The girth of $G$ is at most 5. 
	\end{claim}
	
	\begin{proof}
		We proof by contradiction, assume that $g(G) \geq 6$. If $g(G) \geq 7$, consider a path $u_1 u_2 u_3$ of order 3 and let $v_i$ be a neighbor of $u_i$ in $V(G) \setminus \{u_1, u_2, u_3\}$ for $1 \leq i \leq 3$. Note that $d(v_i, v_j) \geq 3$ for any $1 \leq i < j \leq 3$. If we choose $S$ as a maximal packing containing $\{v_1, v_2, v_3\}$ then the induced subgraph $G[N(S)]$ contains a path of order at least 3, contradicting Proposition \ref{H}. 
		
		Now, we consider the case where $g(G) = 6$. Let $C = u_1 u_2 \ldots u_6 u_1$ be a cycle of length 6 in $G$. For $1 \leq i \leq 6$, let $v_i$ denote the neighbor of $u_i$ outside of $C$. By the girth condition, all these vertices are distinct. If $\{v_1, v_2, u_4\}$ is a packing, then choose $S$ as a maximal packing containing $\{v_1, v_2, u_4\}$, the induced subgraph $G[N(S)]$ contains a path of order at least 3, which contradicts Proposition \ref{H}. Therefore, $\{v_1, v_2, u_4\}$ cannot form a packing of $G$. Since $g(G) = 6$, $\{v_1, v_2, u_4\}$ must be an independent set in $G$. This means that there exist two vertices in $\{v_1, v_2, u_4\}$ that have a common neighbor, indicating that $v_1 v_4 \in E(G)$. By symmetry, we obtain $v_2 v_5 \in E(G)$ and $v_3 v_6 \in E(G)$.
		
		If $S_0 = \{v_1, v_2, v_3\}$ forms a packing, let $S$ be a maximal packing containing $S_0$. Then the induced subgraph $G[N(S)]$ contains a path of order at least 3, contradicting Proposition \ref{H}. Therefore, by the girth condition, $v_1$ and $v_3$ must have a common neighbor, say $w_1$, where $w_1 \notin N[V(C)]$. By symmetry, $v_3$ and $v_5$ have a common neighbor, which must be $w_1$ due to $\Delta(G) \leq 3$. By symmetry, there exists a vertex $w_2$ that is adjacent to each vertex of $\{v_2,v_4,v_6\}$. Then $G$ is the Heawood graph (as illustrated in Figure \ref{TheHeawoodgraph}), with $i(G) = 4$ and $\rho(G) = 2$, which is a contradiction.
		
		Therefore, we get the desired result.

		\begin{figure}[h]
			\centering
			\begin{tikzpicture}
			\coordinate (A) at (0,0);
			\foreach \x in {1,...,14}
			{
				\coordinate (P\x) at (\x*180/7:2cm);
				\node[draw,circle,inner sep=1.2pt,fill] at (P\x){};
			}
			
			\draw (P1)node[above,right]{$u_4$}--(P2)node[above,right]{$u_3$};
			\draw (P3)node[above]{$u_2$}--(P4)node[above]{$u_1$};
			\draw (P5)node[above, left]{$v_1$}--(P6)node[above, left]{$v_4$};
			\draw (P7)node[left]{$w_2$}--(P8)node[below,left]{$v_2$};
			\draw (P9)node[below,left]{$v_5$}--(P10)node[below]{$w_1$};
			\draw (P11)node[below]{$v_3$}--(P12)node[below,right]{$v_6$};
			\draw (P13)node[below,right]{$u_6$}--(P14)node[right]{$u_5$};
			\draw (P1)--(P6);
			\draw (P4)--(P13);
			\draw (P5)--(P10);
			\draw (P7)--(P12);
			\draw (P9)--(P14);
			\draw (P11)--(P2);
			\draw (P3)--(P8);
			\foreach \x/\y in {1/2,2/3,3/4,4/5,5/6,6/7,7/8,8/9,9/10,10/11,11/12,12/13,13/14,14/1}
			{
				
				{	\draw (P\x)--(P\y);}
			}
			
			\end{tikzpicture}
			\caption{The Heawood graph.}
			\label{TheHeawoodgraph}
		\end{figure}
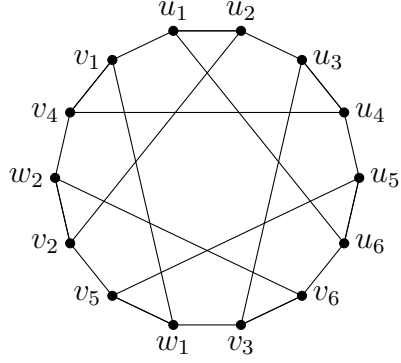

	\end{proof}
	Next, we discuss the cases where the girth of $G$ is 3, 4, and 5, respectively.
	\begin{case}
		If $G$ contains a triangle, then $G\cong H_1$.
	\end{case}
	\begin{proof}
		Assume that $G$ contains a triangle $T = u_1u_2u_3u_1$. If $u_1$ and $u_2$ have a common neighbor $u_3'$ distinct from $u_3$, then by choosing $S$ as a maximal packing with $u_1 \in S$, the path $u_3u_2u_3'$ would be a $P_3$ in $G[N(S)]$, contradicting Proposition \ref{H}. Therefore, no two vertices of $T$ have a common neighbor outside of $T$. Let $v_i$ represent the neighbor of $u_i$ outside of $T$. If $S_0=\{v_1, v_2, v_3\}$ forms a packing, we can choose $S$ as a maximal packing with $S_0 \subseteq S$. However, this leads to $T\subseteq G[N(S)]$, contradicting Proposition \ref{H}. Therefore, there exist two vertices in $\{v_1, v_2, v_3\}$ whose distance is at most 2. Without loss of generality, we can assume that $d(v_1, v_2) \leq 2$.
		
		Now, suppose that $d(v_1,v_2)=1$. If both $v_1v_3$ and $v_2v_3$ are edges of $G$, then the resulting graph has $i(G)=2$ and $\rho(G)=1$, which is a contradiction. Therefore, there exists at least one vertex in $\{v_1,v_2\}$ that is not adjacent to $v_3$. This implies that there exists one neighbor of $v_3$, say $x$, satisfying $\max\{d(u_1,x),d(u_2,x)\} \geq 3$. If $d(u_1,x)\geq 3$ (resp. $d(u_2,x)\geq 3$), then $\{u_1,x\}$ (resp. $\{u_2,x\}$) is a packing of $G$. If we choose $S$ to be a maximal packing containing $\{u_1,x\}$ (resp. $\{u_2,x\}$), then the path $u_{2}u_3v_3$ (resp. $u_{1}u_3v_3$) would be a $P_3$ in $G[N(S)]$, contradicting Proposition \ref{H}.

		Therefore, $d(v_1, v_2) = 2$. We claim that $v_3$ cannot be a common neighbor of $v_1$ and $v_2$. Otherwise, there would exist a neighbor $y$ of $v_1$ such that $y \notin N[V(T)]$, thus, $d(u_3, y) \geq 3$. If we choose $S$ as a maximal packing containing $\{u_3, y\}$, then $u_2u_1v_1v_3$ is a path of order $4$ in $G[N(S)]$, contradicting Proposition \ref{H}. Assume that $v_4$ is a common neighbor of $v_1$ and $v_2$, where $v_4 \notin N[V(T)]$. Then, $v_3v_4 \in E(G)$. If not, choosing $S$ as a maximal packing containing $\{u_3, v_4\}$ as $d(u_3,v_4)\geq 3$, the path $v_1u_1u_2v_2$ would be a $P_4$ in $G[N(S)]$, which contradicts Proposition \ref{H}. We now assert that $v_2v_3 \notin E(G)$. Otherwise, $v_1$ has a neighbor, say $z$, with $d(v_2, z) \geq 3$. In this case, if we choose $S$ as a maximal packing containing $\{v_2, z\}$, the path $v_1v_4v_3$ would be a $P_3$ in $G[N(S)]$, contradicting Proposition \ref{H}. Therefore, $v_2$ has a neighbor $v_5$ such that $v_5 \notin N[V(T)] \cup \{v_4\}$. We claim that $v_3v_5 \in E(G)$. Otherwise, choosing $S$ as a maximal packing with $u_3, v_5$ contained, the path $u_1u_2v_2$ would be a $P_3$ in $G[N(S)]$, contradicting Proposition \ref{H}. Finally, we have $v_1v_5 \in E(G)$. Otherwise, if we choose $S$ as a maximal packing containing $\{u_1, v_5\}$, the path $v_2u_2u_3v_3$ would be a $P_4$ in $G[N(S)]$, contradicting Proposition \ref{H}. Hence, $G \cong H_1$.
	\end{proof}
	
	\begin{case}
		If the girth of $G$ is $4$, then $G\cong H_2$ or $G\cong H_3$.
	\end{case}
	\begin{proof}
		{\bf Subcase 1.} $G$ contains $K_{2,3}$ as a subgraph.
		
		Let $G$ be a graph containing a $K_{2,3}$ with partite sets $\{u_1,u_2\}$ and $\{v_1,v_2,v_3\}$. Since $G$ contains no triangle, every copy of $K_{2,3}$ in $G$ is induced. For each $1\leq i\leq 3$, let $w_i$ be the neighbor of $v_i$ outside $\{u_1,u_2\}$.
		
		First, if $w_1=w_2=w_3$, then $G\cong K_{3,3}$, namely, $G\cong H_2$.
		
		Next, suppose that $w_1=w_2\neq w_3$. In this case, $w_1w_3\in E(G)$. If not, $d(v_1,w_3)\geq 3$. By choosing $S$ as a maximal packing containing $v_1$ and $w_3$, we observe that $u_1v_3u_2$ forms a path of order $3$ in $G[N(S)]$, which contradicts Proposition \ref{H}. Let $x$ represent the neighbor of $w_3$ distinct from $v_3$ and $w_1$. Now $d_G(v_2,x)=3$. By choosing $S$ as a maximal packing containing $\{v_2,x\}$, we deduce that $w_1w_3$ is an edge in $G[N(S)]$. By Proposition \ref{prop3}, if $w_1\in A$ and $w_3\in Z^*$, then $v_3\in Z$ and $u_1,u_2\in \overline{Z}$, which is a contradiction with $\{u_1,u_2\}\subseteq N(S)$. Therefore, $w_1\in Z^*$ and $w_3\in A$. Thus, $v_1\in Z$, and $v_1$ has two neighbors, namely $u_1$ and $u_2$, in $\overline{Z}$ by Proposition \ref{prop3}. This contradicts $\{u_1,u_2\}\subseteq N(S)$.
		
		Finally, suppose that $w_1\neq w_2\neq w_3$. Since $G$ contains no triangle, by symmetry, we may assume that $w_1w_2\notin E(G)$. Now, $d_G(w_1,v_2)=3$. By selecting $S$ as a maximal packing containing $w_1$ and $v_2$, we obtain that $u_1v_1u_2$ is a path of order $3$ in $G[N(S)]$, which contradicts Proposition \ref{H}.

		{\bf Subcase 2.} $G$ dose not contain $K_{2,3}$ as a subgraph.
		
		Suppose that $G$ contains a cycle $C=u_1u_2u_3u_4u_1$ of length 4. Let $v_i$ denote the neighbor of $u_i$ outside $C$ for each $1\leq i\leq 4$, and let $I=\{v_1,v_2,v_3,v_4\}$.
		Since $G$ contains neither a $K_3$ nor a $K_{2,3}$, the vertices in $I$ are pairwise distinct. 
		If $d(v_1,u_3)\geq 3$, then consider a maximal packing $S$ containing $\{v_1,u_3\}$. Note that $u_2u_1u_4$ forms a path of order 3 in $G[N(S)]$, which contradicts Proposition \ref{H}. Therefore, $v_1v_3\in E(G)$. By symmetry, $v_2v_4\in E(G)$.
		
		If $d(v_1,v_2)=3$, then choosing a maximal packing $S$ containing $\{v_1,v_2\}$, so $u_1u_2$ is an edge in $G[N(S)]$.
		By symmetry, we only consider the case that $u_1\in A$ and $u_2\in Z^*$. Therefore, it follows from Proposition \ref{prop3} that $u_3\in Z$, and $u_3$ has two neighbors, namely $v_3$ and $u_4$, in $\overline{Z}$. This contradicts that $v_3\in N(S)$. Hence, by symmetry, the distance of any two vertices in $I$ is at most 2.
		
		If $G$ contains none of the edges in $\{v_1v_2,v_2v_3,v_3v_4,v_4v_1\}$, then there exists a vertex outside $V(C)\cup I$, say $w$, that is adjacent to both $v_1$ and $v_2$. Since $d_G(v_1,v_4)=2$ and $d_G(v_2,v_3)= 2$, this implies that $w$ is adjacent to $v_3$ and $v_4$, which contradicts that $\Delta(G)\leq 3$. 
		Hence, by symmetry, we may assume that $v_1v_2\in E(G)$.  
		
		If $v_3$ and $v_4$ are not adjacent, there is a new vertex, say $w'$, that is adjacent to both $v_3$ and $v_4$. 
		Clearly, $d(u_1,w')=3$. Choosing $S$ as a maximal packing containing $\{u_1,w'\}$, we find that $v_1v_3$ is an edge in $G[N(S)]$.
		By Proposition \ref{prop3}, if $v_1\in A$ and $v_3\in Z^*$, then we obtain that $u_3\in Z$ and $u_3$ has two neighbors, namely $u_2$ and $u_4$, in $\overline{Z}$, which conflicts with $\{u_2,u_4\}\subseteq N(S)$.
		Similarly, if $v_1\in Z^*$ and $v_3\in A$, then $v_2\in Z$ and $v_2$ is adjacent to two vertices of $\overline{Z}$. This leads to $\{u_2,v_4\}\subseteq \overline{Z}$, which conflicts with $\{u_2,v_4\}\in N(S)$.
		Hence, $v_3v_4\in E(G)$, which implies that $G\cong H_3$.
	\end{proof}
	
	\begin{case}
		If the girth of $G$ is $5$, then $G\cong H_4$.
	\end{case}
	
	\begin{proof}		
		Let $|V(G)|=n$.
		If $n \ge 12$, then by Lemma \ref{Favathm}, $\rho(G) \ge \frac{n}{8}$, and by Lemma \ref{Henningconj}, $i(G) \le \frac{3n}{8}$. 
		Therefore, $i(G) \le \frac{3n}{8} \le 3\rho(G)$. Combining this condition with $i(G) = 3\rho(G)$, we find that $\rho(G) = \frac{n}{8}$ and $n$ is divisible by 8. However, by Theorem \ref{packingtheorem}, this is not possible. Therefore, we only need to consider the case where $n \le 10$ since $G$ is cubic.
		
		Let $C = u_0u_1u_2u_3u_4u_0$, and let $v_i$ be the neighbor of $u_i$ outside $C$ for each $0 \le i \le 4$. Since $g(G) = 5$, the vertices in $\{v_i: 0 \le i \le 4\}$ are pairwise distinct. If $v_iv_{i+1} \in E(G)$ (with subscripts modulo $5$) for some $0 \le i \le 4$, then there is a cycle of length 4, a contradiction. 
		Therefore, $v_iv_{i+1} \notin E(G)$ (with subscripts modulo $5$) for each $0 \le i \le 4$.
		Moreover, since $G$ is cubic, we have $N(v_i) = \{u_i, v_{i-2}, v_{i+2}\}$ (with subscripts modulo $5$) for each $0 \le i \le 4$. Hence, $G\cong H_4$.
	\end{proof}

\end{document}